\documentclass[reqno,12pt]{amsart} 
\usepackage[margin=2.85cm,footskip=1cm]{geometry}

\usepackage{amsmath} 

\usepackage{amssymb}      

\usepackage{amscd}      

\usepackage{amsthm}      
   
\usepackage{epsfig}      

\usepackage{amstext}      

\usepackage[all]{xy} 
\CompileMatrices 

\usepackage[latin1]{inputenc}

\usepackage{multirow}

\newcommand{\Ad}{\operatorname{Ad}}

\newcommand{\id}{\operatorname{id}}

 \newcommand{\Ext}{\operatorname{Ext}}

\newcommand{\C}{\operatorname{\mathbb{C}}}
\newcommand{\N}{\operatorname{\mathbb{N}}}

   \theoremstyle{plain}
   \newtheorem{thm}{Theorem}[section]
   
   \newtheorem{lemma}[thm]{Lemma}  
   
   \theoremstyle{definition}
   
   \newtheorem{defn}[thm]{Definition}
   
   \theoremstyle{remark}

   \numberwithin{equation}{section}

\title[Free
product of amenable groups]{Extensions of the reduced group $C^*$-algebra of a free
product of amenable groups}
  
\author{Jonas Andersen Seebach and Klaus Thomsen}

        \date{\today}

\date{\today}

\email{matkt@imf.au.dk}
\address{Institut for matematiske fag, Ny Munkegade, 8000 Aarhus C, Denmark}

\begin{document}

\maketitle

\begin{abstract} We prove that the unitary equivalence classes of
  extensions of $C^*_r(G)$ by any $\sigma$-unital stable
  $C^*$-algebra, taken modulo extensions which split via an asymptotic
  homomorphism, form a group which can be calculated from the
  universal coefficient theorem of $KK$-theory when $G$ is a free
  product of a countable collection of countable amenable groups.
\end{abstract}

\section{Introduction}

The stock of examples of $C^*$-algebras for which the semi-group of
extensions by the compact operators is not a group is still
growing. The latest newcomers consist of a series of reduced free products
of nuclear $C^*$-algebras, cf. \cite{HLSW}. This stresses the necessity of
finding a way
to handle the many extensions without
inverses. In joint work with Vladimir Manuilov the second-named author
has proposed a way to amend the definition of the semi-group of
extensions of a $C^*$-algebra by a stable $C^*$-algebra in
such a way that nothing is changed in the case of nuclear algebras
where the usual theory already works perfectly, and such that at least
some of the extensions which fail to have inverses in
the usual sense become invertible in the new, slightly weaker
sense. This new semi-group grew out of investigations of the relation
between the $E$-theory of Connes and Higson and the theory of
$C^*$-extensions, \cite{MT1}. The change consists merely in
trivializing not only the split extensions, but also the asymptotically split
extensions; those
for which there is an asymptotic
homomorphism consisting of right inverses for the quotient map,
cf. \cite{MT1}. When an extension can be made asymptotically split by
addition of another extension we say that the extension is
semi-invertible, and the resulting group of semi-invertible extensions, taken
modulo asymptotically split extensions, is an abelian group with a
close connection to the $E$-theory of Connes and Higson, \cite{CH}. In some, but not all cases where the usual semi-group
of extensions is not a group the alternative definition does give a
group; i.e. all extensions are semi-invertible, cf. \cite{MT1},\cite{Th1},\cite{MT3}. Specifically, in
\cite{MT1} this was shown to be the case when the quotient is a
suspended $C^*$-algebra and in \cite{Th1} when the
quotient is the reduced group $C^*$-algebra $C^*_r\left(\mathbb F_n\right)$ of a free group with finitely
many generators, and the ideal is the $C^*$-algebra $\mathbb K$ of
compact operators. This gave the first example of a unital $C^*$-algebra
for which all extensions by the compact operators are semi-invertible,
but not all invertible; by the result of Haagerup and Thorbjørnsen,
\cite{HT}, there are non-invertible extensions of $C^*_r\left(\mathbb
  F_n\right)$ by $\mathbb K$ when $n \geq 2$. The purpose of the
present note is to show that the situation in \cite{Th1} is not
exceptional at all. This is done by showing that all extensions of a reduced group
$C^*$-algebra $C^*_r(G)$ by any stable $\sigma$-unital $C^*$-algebra
is semi-invertible when $G$ is
the free product of a countable collection of discrete countable and amenable groups. The basic idea of the proof is identical to that employed in
\cite{Th1}. The crucial improvement over the argument from \cite{Th1}
is that the explicitly given homotopy of representations of $\mathbb F_n$
from \cite{C} is replaced by results of Dadarlat and Eilers from \cite{DE}. The
pairing in the first variable of the usual extension group $\Ext^{-1}$ with
$KK$-theory and Cuntz' results on K-amenability from \cite{C} remain
key ingredients.  

In \cite{Th1} the inverse of an extension, modulo asymptotically split
extensions, could be taken to be invertible in the usual sense,
i.e. to admit a completely positive contractive splitting. This turns
out to be possible also in the more general situation considered
here, and as a consequence it follows that the obvious
map from the usual KK-theory group $\Ext^{-1}(C^*_r(G),B)$ to the
group of all extensions, taken modulo asymptotically split extensions,
is surjective. By combining results of Cuntz, Tu and Thomsen it
follows that $C^*_r(G)$ satisfies the universal coefficient theorem of
Rosenberg and Schochet, and from this it follows easily that the map
is also injective. Hence
the group of extensions of $C^*_r(G)$ by $B$, taken modulo the
asymptotically split extensions, can be calculated from K-theory by
use of the UCT.

\section{The results}

\begin{thm}\label{main} Let $(G_i)_{i\in \N}$ be a countable collection of
  discrete countable amenable groups and let $G = \star_i G_i $ be
  their free product. Let $B$ be a stable $\sigma$-unital
  $C^*$-algebra. For every extension $\varphi : C^*_r(G) \to Q(B)$
  there is an invertible extension $\varphi' : C^*_r(G) \to Q(B)$ such
  that $\varphi \oplus \varphi'$ is asymptotically split.
\end{thm}

More explicitly the conclusion is that there is an extension
$\varphi'$, a completely positive contraction $\psi : C_r^*(G) \to
M(B)$ and an asymptotic $*$-homomorphism $\pi = \left(\pi_t\right)_{t
  \in [1,\infty)} : C^*_r(G) \to M(B)$, in the sense of Connes and
Higson, cf. \cite{CH},  such that $\varphi' = q_B \circ
\psi$ and $\varphi \oplus \varphi' = q_B\circ \pi_t$ for all $t \in
[1,\infty)$, where $q_B : M(B) \to Q(B)$ is the quotient map.

If only one of the $G_i$'s in Theorem \ref{main} is non-trivial or if $G =
\mathbb Z_2 \star \mathbb Z_2$, the conclusion of the theorem is
trivial and can be improved because $G$ is then amenable. It seems very
plausible that such cases are exceptional; indeed it follows from
\cite{HLSW} that there is a non-invertible extension of $C^*_r(G)$ by
the compact operators whenever $G$ is the free product of finitely
many non-trivial groups each of which is either abelian or finite and
$G \neq \mathbb Z_2 \star \mathbb Z_2$.

\bigskip

As in \cite{Th1} we will prove Theorem \ref{main} by use of results
from \cite{MT2}. Recall that two extensions $\varphi,\varphi' : A \to
Q(B)$ are \emph{strongly homotopic} when there is a $*$-homomorphism
$A \to C[0,1]\otimes Q(B)$ giving us $\varphi$ when we evaluate at $0$
and $\varphi'$ when we evaluate at 1. By Lemma 4.3 of \cite{MT2} it
suffices then to
establish the following

\bigskip

\begin{thm}\label{main2} Let $(G_i)_{i\in \N}$ be a countable
  collection of
  discrete countable amenable groups and let $G = \star_i G_i $ be
  their free product. Let $B$ be a stable $\sigma$-unital
  $C^*$-algebra. For every extension $\varphi : C^*_r(G) \to Q(B)$
  there is an invertible extension $\varphi' : C^*_r(G) \to Q(B)$ such
  that $\varphi \oplus \varphi'$ is strongly homotopic to a split extension.
 \end{thm}

For the proof of Theorem \ref{main2} we will need the following notion.

\begin{defn} \label{weak}
Let $ A $ be a $ C^* $-algebra and $ \varphi, \psi $ be $ * $-representations 
 of $ A $ on some Hilbert spaces.
 Then $ \varphi $ is weakly contained in $ \psi $ if $ \ker\psi
 \subseteq \ker\varphi $.
\end{defn}

If $ \sigma, \pi $ are unitary representations of a locally compact group then
 $ \sigma $ is weakly contained in $ \pi $ if and only if
 the representation of the full group $ C^* $-algebra
 corresponding to $ \sigma $ is weakly contained in the
 representation corresponding to $ \pi $. An equivalent definition of weak containment in this case, is that every positive definite function associated to $ \sigma $
 can be approximated uniformly on compact subsets by finite sums of positive definite functions associated
 to $ \pi $. See sections 3.4 and 18.1 of \cite{Di} for details.

A proof of the following lemma can be found in \cite{BHV}.

\begin{lemma}\label{cuntz2}
Let $\sigma,\pi$ be unitary representations of a locally compact group. Assume that $\sigma$
is weakly contained in the left regular representation $ \lambda $. It follows
that $\sigma \otimes \pi$ is weakly contained in $\lambda$.
\end{lemma}

For any discrete group $G$ we denote in the following the canonical
surjective $*$-homomor\-phism $C^*(G) \to C^*_r(G)$ from the full to the
reduced group $C^*$-algebra by $\mu$.

Besides the main results of \cite{C} we shall also need the following
technical lemma concerning Cuntz' K-amenability. See \cite{C} for the proof.

\begin{lemma} \label{K-amen}
Let $ H $ be an infinite-dimensional separable Hilbert space and let $ G $ be a countable discrete K-amenable group. Then there exist $*$-homomorphisms
$ \sigma, \sigma_0:C_r^*(G) \to B(H) $ such that $\sigma \circ \mu,
h_t \oplus \sigma_0 \circ \mu :C^*(G) \to B(H)  $ are unital, $\sigma
\circ \mu(a) - \left(h_t \oplus \sigma_0 \circ \mu\right)(a) \in
\mathbb K$ for all $a \in C^*(G)$, and
$\left[\sigma \circ \mu,h_t \oplus \sigma_0 \circ \mu \right] = 0$
in $KK\left(C^*(G), \mathbb C\right)$,
 where $h_t : C^*(G) \to \mathbb C
\subseteq B(H)$ is the $*$-homomorphism going with the trivial
one-dimensional representation $t$ of $G$.
\end{lemma}


Our proof of Theorem \ref{main2} uses the notion of absorbing and unitally absorbing
$*$-homomorphisms. We refer to \cite{Th2} for the definition and the
proof that they exist in the generality required in the
argument. Furthermore, we shall need the following lemma which is a
unital version of Lemma 2.2 in \cite{Th3}. The proof is the same.

\begin{lemma}\label{2.2} Let $A$ be a separable unital $C^*$-algebra, $D
  \subseteq A$ a unital nuclear $C^*$-subalgebra and $B$ a stable
  $\sigma$-unital $C^*$-algebra. Let $\pi : A \to M(B)$ be a unitally
  absorbing $*$-homomorphism. It follows that $\pi|_D : D \to M(B)$ is
  unitally absorbing.
\end{lemma}

The next lemma gives us the appropriate substitute for the homotopy of
representations of $\mathbb F_n$ which was a crucial tool in \cite{Th1}. 

\begin{lemma}\label{correction} Let $(G_i)_{i\in \N}$ be a countable collection of
  discrete countable amenable groups and let $G = \star_i G_i $ be
  their free product. Let $\mu : C^*(G) \to
  C^*_r(G)$ be the canonical surjection and let $h_t : C^*(G) \to
  \mathbb C$ be the character corresponding to the trivial
  one-dimensional representation of $G$. There is then a separable infinite-dimensional
  Hilbert space $H$, unital
  $*$-homomorphisms $\sigma,\sigma_0 : C^*_r(G) \to B(H)$ and a
  path $\nu_s : C^*(G) \to B(H), s \in [0,1]$, of unital
  $*$-homomorphism such that
\begin{enumerate}
\item[a)] $\nu_0 = \sigma \circ \mu$;
\item[b)] $\nu_1 = h_t \oplus\sigma_0\circ \mu$;
\item[c)] $\nu_s(a) - \nu_0(a) \in \mathbb K, \ a \in C^*(G), s \in [0,1]$, and
\item[d)] $s \mapsto \nu_s(a)$ is continuous for all $a \in C^*(G)$.
\end{enumerate} 
\begin{proof} Being amenable $ G_i $ has the
Haagerup Property. See the discussion in 1.2.6 of \cite{CCJJV}. It
follows then from Propositions 6.1.1 and 6.2.3 of \cite{CCJJV} that
also $ G $ has the Haagerup Property. Since the Haagerup Property
implies K-amenability by \cite{Tu} we conclude that $G$ is
K-amenable. We can therefore pick
$*$-homomorphisms $\sigma, \sigma_0 : C^*_r(G) \to B(H)$ as in Lemma \ref{K-amen}. 
By adding the same unital and injective $*$-homomorphism to $\sigma$
and $\sigma_0$ we can arrange that both $\sigma$ and $\sigma_0$ are
injective and have no non-zero compact operator in their range. Since
$\mu|_{C^*_r(G_i)} : C^*(G_i) \to C^*_r(G_i)$ is injective it
follows then that $\sigma \circ \mu|_{C^*(G_i)}$ and $(h_t \oplus \sigma_0
\circ \mu)|_{C^*(G_i)}$ are admissible in the sense of Section 3 of
\cite{DE} for each $i$. Thus Theorem 3.12 of
\cite{DE} applies to show that there is a
norm-continuous path $u^i_s, s\in [1,\infty)$, of unitaries in $1 + \mathbb K$ such that$$
\lim_{s \to \infty}
\left\|\sigma\circ \mu|_{C^*(G_i)}(a)  -
  u^i_s \left(h_t \oplus \sigma_0\circ
   \mu\right)|_{C^*(G_i)}(a){u^i_s}^* \right\| = 0
$$ 
for all $a \in
C^*(G_i)$ and 
$$
\sigma \circ \mu|_{C^*(G_i)}(a)  -
  u^i_s \left(h_t \oplus \sigma_0\circ
   \mu\right)|_{C^*(G_i)}(a){u^i_s}^* \in \mathbb K
$$ 
for all $a
  \in C^*\left(G_i\right)$ and all $s \in [1,\infty)$. 
Since the
  unitary group of $1 + \mathbb K$ is connected in norm there are
  therefore norm-continuous paths of unital
$*$-homomorphisms $\nu^j_s : C^*(G_j) \to B(H), \ s \in [0,1], \ j \in
\mathbb N$, such
that
\begin{enumerate}
\item[aj)] $\nu^j_0 = \sigma \circ \mu|_{C^*(G_j)}$;
\item[bj)] $\nu^j_1 = \sigma_0\circ \mu|_{C^*(G_j)} \oplus
  h_t|_{C^*(G_j)}$;
\item[cj)] $\nu^j_s(a) - \nu^j_0(a) \in \mathbb K, \ a \in C^*(G_j),
  \ s\in [0,1]$,
\end{enumerate}
for each $j$. The universal property of the free product construction
gives us then a path of unital
$*$-homomorphisms $\nu_s : C^*(G) \to B(H),  s \in [0,1]$, with the
stated properties, a)-d).

\end{proof}
\end{lemma}

\begin{lemma} \label{fuld2}
In the setting of Theorem \ref{main} it holds that every extension
$\varphi : C^*(G) \to Q(B)$ of $C^*(G)$ by $B$ is invertible. If
$\varphi $ is unital, it is invertible in
the semi-group of unitary equivalence classes of unital extensions,
modulo the unital split extensions.

\end{lemma}
\begin{proof} Assume first that $ \varphi $ is unital. For each $i \in \N$ the $C^*$-algebra $C^*_r(G_i) = C^*(G_i)$ is
  nuclear and hence the unital extensions $ \varphi_i = \varphi|_{C^*(G_i)} :
  C^*\left(G_i\right) \to Q(B) $ are all invertible. There are
  therefore unital extensions $\psi_i : C^*\left(G_i\right) \to Q(B)$
  and $*$-homomorphisms $\pi_i : C^*\left(G_i\right) \to M(B)$ such that $\varphi_i
  \oplus \psi_i = q_B \circ \pi_i, \ i \in \N$. Let $ \omega_i:C^*(G_i) \to \C $ denote the $*$-homomorphism corresponding to the trivial unitary representation of $ G_i $. By replacing $ \pi_i $ with 
$ \pi_i + \omega_i \pi_i(1)^\perp $ we may assume that $ \pi_i $ is unital. The universal
  property of the free product gives us a unital extension $\psi = \star_i \psi_i
  : C^*(G) \to Q(B)$ and a unital
  $*$-homomorphism $\pi = \star_i \pi_i : C^*(G) \to M(B)$. Since $\varphi \oplus \psi
  = q_B \circ \pi$, this completes the proof of the unital case.

Now let $ \varphi $ be a general extension. Again consider $ \varphi_i = \varphi|_{C^*(G_i)} :
  C^*\left(G_i\right) \to Q(B) $. Then $ \varphi_i(1)=\varphi_j(1)=p $ for all $ i,j \in \N $, 
so the extensions $ \tilde{\varphi}_i:= \varphi_i + \omega_i p^\perp $ are all unital. As above we get 
$\psi : C^*(G) \to Q(B)$ and a unital
  $*$-homomorphism $\pi : C^*(G) \to M(B)$ such that $\left(\star_i \tilde{\varphi}_i\right) \oplus \psi
  = q_B \circ \pi$. Since $ \star_i \tilde{\varphi}_i$ and $\varphi \oplus \left(\star_i \omega_i p^{\perp}\right) $ are equal in $ \Ext(C^*(G),B) $ this completes the proof.
\end{proof}




\begin{proof}[Proof of Theorem \ref{main2}] In order to control the
  image of the unit for the extensions we consider, we need a result of
  Skandalis which we first describe. Note that the unital inclusion $i
: \mathbb C \to C^*(G)$ has a left-inverse $h_t : C^*(G) \to \mathbb C$
given by the trivial one-dimensional representation $t$. Therefore the map
$$
i^* : \Ext^{-1}(C^*(G),SB) \to \Ext^{-1}(\mathbb C,SB) = K_0(B)
$$
is surjective. We put this into the six-term exact sequence
of Skandalis, 10.11 in \cite{S}, whose proof can be found in
\cite{MT4}. Using the notation from \cite{MT4} we obtain the following commuting diagram with exact rows:
\begin{equation}\label{diagram}
\begin{xymatrix}{
0 \ar[r] & \Ext_\text{unital}^{-1}\left(C^*(G),B\right) \ar[r] &
\Ext^{-1}(C^*(G),B) \ar[r] &  K_0(Q(B)) \ar@{=}[d]  \\
{} &  \Ext_\text{unital}^{-1}\left(C_r^*(G),B\right) \ar[r]\ar[u]^-{\mu^*}  &
\Ext^{-1}(C_r^*(G),B) \ar[r] \ar[u]^-{\mu^*} &  K_0(Q(B))  
}
\end{xymatrix}
\end{equation}
$G$ is K-amenable as observed in the proof of Lemma \ref{correction}. By \cite{C} this implies that $\mu^* : \Ext^{-1}\left(C^*_r(G),B\right) \to
  \Ext^{-1}\left(C^*(G),B\right)$ is an isomorphism.

Let $ \varphi: C_r^*(G) \to Q(B) $ be a unital extension. Let $\pi_1: C^*_r(G) \to Q(B)$ be a unitally
absorbing split extension (whose existence is guaranteed by
\cite{Th2}) and set $\varphi' = \varphi \oplus \pi_1$. It follows from
Lemma \ref{fuld2} and a diagram chase in (\ref{diagram}) that there
is an invertible unital extension $\varphi'' : C^*_r(G) \to Q(B)$ such that
\begin{equation}\label{eq102}
\left[\varphi' \circ \mu \oplus \varphi''\circ \mu\right] = 0
\end{equation} 
in $\Ext^{-1}_\text{unital}\left(C^*(G),B\right)$. Since $C^*(G_i)$ is nuclear $\mu|_{C^*(G_i)} : C^*(G_i) \to C^*_r(G_i), i \in \N$, is a $*$-isomorphism and it
follows from Lemma \ref{2.2} that
$\pi_1|_{C^*_r(G_i)} : C^*_r(G_i) \to Q(B)$ is unitally absorbing for each $i
\in \N$. Hence $\pi_1 \circ
\mu|_{C^*(G_i)} : C^*(G_i) \to Q(B)$ is a unitally absorbing split
  extension. It follows therefore from (\ref{eq102}) that
  $\left(\varphi' \circ \mu \oplus \varphi''\circ
    \mu\right)|_{C^*(G_i)}$ is a unitally split extension for each $i$. As in
  the proof of Lemma \ref{fuld2} this implies that $\varphi' \circ \mu
  \oplus \varphi''\circ \mu$ is unitally split. There is therefore a
  unitary representation $R : G \to M(B)$ such that 
\begin{equation}\label{eq107}
q_B \circ h_R = \varphi' \circ \mu
  \oplus \varphi''\circ \mu,
\end{equation}
where $h_R : C^*(G) \to M(B)$ is the $*$-homomorphism defined by $R$.


Consider the homotopy $\nu_s$ from Lemma \ref{correction}. Let $\nu'_s : G \to B(H)$ be the unitary representation defined by $\nu_s$ so
that $\nu_s = h_{\nu'_s}$. It follows from the property a) of Lemma
\ref{correction} that $\nu'_0$ is weakly
contained in the left-regular representation of $G$ and from b) that
$\nu'_1$ is a direct sum $t \oplus \lambda_0$ where $\lambda_0$ is a
representation of $G$ which is weakly contained in the left-regular
representation of $G$. Consider the unitary
representations
$$
R \otimes \nu'_s  : G \to M(B)\otimes B(H) \subseteq M(B \otimes \mathbb
K), \ s \in [0,1].
$$
Then $q_{B \otimes \mathbb K} \circ h_{R \otimes \nu'_s} : C^*(G)
\to Q(B\otimes \mathbb K), \ s \in [0,1]$, is a norm-continuous path
of extensions. Note that 
$$
 q_{B \otimes \mathbb K} \circ h_{R \otimes \nu'_1}  = q_{B \otimes
  \mathbb K} \circ h_{R \otimes t} \oplus q_{B \otimes
  \mathbb K} \circ h_{R \otimes \lambda_0}  = (\varphi' \oplus
\varphi'')\circ \mu \oplus q_{B \otimes
  \mathbb K} \circ h_{R \otimes \lambda_0}.
$$
Since $R \otimes \nu'_0$ and $R \otimes \lambda_0$ are
weakly contained in the left-regular representation of $G$ by Lemma
\ref{cuntz2} it follows from an argument almost identical with one
used in \cite{Th1} that
each $q_{B \otimes \mathbb K} \circ h_{R \otimes\nu'_s}$ factors
through $C_r^*(G)$ and gives us a strong homotopy connecting the split
extension $q_{B \otimes \mathbb K} \circ h_{R \otimes \nu'_0} :
C^*_r(G) \to Q(B \otimes \mathbb K)$ to the direct sum $ \varphi' \oplus
\varphi''  \oplus q_{B \otimes
  \mathbb K} \circ h_{R \otimes \lambda_0}$. For completeness we include the argument: Let $s \in [0,1]$ and $x =
\sum_{j} c_j g_j  \in
\mathbb CG$, where $c_j \in \mathbb C$ and $g_j
\in G$. Then
\begin{equation}\label{e7}
h_{R \otimes \nu'_s}(x) = \sum_j c_j R\left(g_j\right)\otimes
\nu'_0\left(g_j\right) + \sum_j c_j R\left(g_j\right) \otimes
\Delta\left(g_j\right),
\end{equation}
where $\Delta\left(g_j\right) = \nu'_s\left(g_j\right) -
\nu'_0\left(g_j\right)$. Note that $\Delta\left(g_j\right) \in
\mathbb K$ by c). Since $\nu'_0$ is weakly contained in the left regular
representation we can use Lemma \ref{cuntz2} to conclude that
$\left\| \sum_j c_j R\left(g_j\right)\otimes
\nu'_0\left(g_j\right)\right\| \leq \left\|x\right\|_{C^*_r\left(G\right)}$ and hence
\begin{equation*}\label{e2} 
\left\|q_{B\otimes \mathbb K}\left(\sum_j c_j R\left(g_j\right)\otimes
\nu'_0\left(g_j\right)\right)\right\| \leq  \left\|x\right\|_{C^*_r\left(G\right)}.
\end{equation*}
To handle the second term in (\ref{e7}) note that $M(B)\otimes
\mathbb K/B \otimes \mathbb K \simeq Q(B) \otimes \mathbb K$ so
\begin{equation*}
\left\| q_{B \otimes \mathbb K}\left(\sum_j c_j R\left(g_j\right) \otimes
\Delta\left(g_j\right)\right) \right\| = \left\| \sum_j c_j
\left(\varphi'  \oplus \varphi'' \right)\left(g_j\right) \otimes \Delta\left(g_j\right)\right\|_{Q(B)
\otimes \mathbb K} .
\end{equation*}
Since $\varphi' \oplus \varphi'': C^*_r\left( G\right) \to
Q(B)$ is injective (because $\varphi'$ contains the unitally absorbing
split extension $\pi_1$) and
$(\varphi' \oplus \varphi'') \otimes \id_{\mathbb K}$ isometric,
\begin{equation*}
\left\| \sum_j c_j
\left(\varphi' \oplus \varphi''\right)\left(g_j\right) \otimes \Delta\left(g_j\right)\right\|_{Q(B)
\otimes \mathbb K} = \left\| \sum_j c_j \lambda\left(g_j\right)  \otimes
\Delta\left(g_j\right)\right\|_{C^*_r\left(G\right)
\otimes \mathbb K} .
\end{equation*}
And 
\begin{equation*}\label{e3}
\begin{split}
&\left\| \sum_j c_j\lambda\left(g_j\right)  \otimes
\Delta\left(g_j\right)\right\|_{C^*_r\left(G\right)
\otimes \mathbb K} = \\
&\left\| \sum_j c_j \lambda\left(g_j\right)  \otimes
\nu'_s\left(g_j\right) -  \sum_j c_j \lambda\left(g_j\right)  \otimes
\nu'_0\left(g_j\right) \right\| \leq 2\left\|x\right\|_{C^*_r(G)},
\end{split}
\end{equation*}
by Fell's absorbtion principle or Lemma \ref{cuntz2}. Inserting these estimates into
(\ref{e7}) yields the conclusion that
$$
\left\|q_{B \otimes \mathbb K} \circ h_{R \otimes \nu'_s}(x) \right\|
\leq 3 \left\|x\right\|_{C^*_r(G)} ,
$$
proving that $q_{B \otimes \mathbb K} \circ h_{R \otimes \nu'_s}$
factors through $C^*_r(G)$ as claimed.

It remains to reduce the general case of a possibly non-unital
extension to the case of a unital extension. Let $ \varphi: C^*_r(G)
\to Q(B) $ be an arbitrary extension. From Lemma \ref{fuld2} and K-amenability we get 
an invertible extension $\varphi':
C^*_r(G) \to Q(B)$ such that $\left[\varphi \circ \mu \oplus \varphi'
  \oplus \mu\right] = 0$ in $\Ext^{-1}(C^*(G),B)$. In particular,
$$
p = \left(\varphi \circ \mu \oplus \varphi' \circ \mu\right)(1) 
$$
is a projection which represents $0$ in $K_0(Q(B))$. Since $K_0(M(B)) = 0$ we see $[1-p] +
[p] = [1] = 0$ in $K_0(Q(B))$ so we find that
also $p^{\perp} = 1-p$ represents $0$ in $K_0(Q(B))$.
Since $ M_k(Q(B)) \simeq Q(B) $ for all $ k $ this implies that
$$
p \oplus 1 \ \sim \ 0 \oplus 1 \ \text{and } p^{\perp} \oplus 1 \ \sim
\ 0 \oplus 1
$$
in $M_2(Q(B))$, where $\sim$ is Murray-von Neumann equivalence. It follows that
$$
p \oplus 1 \oplus 0 \ \sim \ 1 \oplus 0 \oplus 0 \ \text{and} \
p^{\perp} \oplus 0 \oplus 1 \ \sim \  0 \oplus 1 \oplus 1
$$
in $M_3(Q(B))$. 
So there is a unitary $w \in M_4(Q(B))$
  contained in the connected component of the unit in the unitary
  group of $M_4(Q(B))$ such that
$$
w\left(p \oplus 1 \oplus 0 \oplus 0\right)w^* = 1 \oplus 0 \oplus 0
\oplus 0 .
$$
Let $\chi : C^*_r(G) \to Q(B)$ be a unital split extension. It follows
that 
\begin{equation}\label{a1}
w \left( (\varphi \circ \mu \oplus \varphi'\circ \mu) \oplus \chi
  \circ \mu \oplus 0 \oplus 0\right)w^* = \psi_0 \oplus 0 \oplus 0 \oplus 0
\end{equation}
for some unital extension $\psi_0 : C^*(G) \to Q(B)$. It follows from
(\ref{a1}) that $\psi_0$ factors through $C^*_r(G)$, i.e. there is
a unital extension $\psi: C^*_r(G) \to Q(B)$ such that $\psi_0 = \psi \circ \mu$. Via an isomorphism $M_4(Q(B)) \simeq M_2(Q(B))$ which
leaves the upper lefthand corner unchanged, we see that there is
an invertible extension $\varphi'' : C^*_r(G) \to Q(B)$ and a unitary $u$ in the connected component of
$1$ such that 
$$
\Ad u  \circ ( \varphi \oplus \varphi'' )  = \psi \oplus 0  
$$
as $*$-homomorphisms $C^*_r(G) \to Q(B)$. Since $\psi$ is unital the
first part of the proof gives us an invertible (unital) extension $\psi' : C^*_r(G) \to Q(B)$ such
that $\psi \oplus \psi' $ is strongly homotopic to a split
extension. Since
$$
\varphi \oplus \varphi'' \oplus \psi'  = \Ad (u^* \oplus 1) \circ (\psi
\oplus 0 \oplus \psi'),
$$
we conclude that $\varphi \oplus \varphi'' \oplus \psi'$ is
strongly homotopic to a split extension. Note that $\varphi'' \oplus
\psi'$ is invertible.
\end{proof}

Let $A$ be a separable $C^*$-algebra and $B$ a stable $\sigma$-unital $C^*$-algebra.
Following \cite{MT1} we let $\Ext^{-1/2}(A,B)$ denote the group of unitary equivalence
classes of semi-invertible extensions
of $A$ by $B$. There is then an obvious map
$$
\Ext^{-1}(A,B) \to \Ext^{-1/2}(A,B)
$$
which in \cite{Th1} was shown to be an isomorphism when $B = \mathbb K$ and $A =
C^*_r(\mathbb F_n)$. We can now
extend this conclusion as follows.

\begin{thm}\label{mainiso} Let $(G_i)_{i\in \N}$ be a countable collection of
  discrete countable amenable groups and let $G = \star_i G_i $ be
  their free product. Let $B$ be a stable $\sigma$-unital
  $C^*$-algebra. It follows that $C^*_r(G)$ satisfies the UCT and that the
  natural map
  $\Ext^{-1}(C^*_r(G),B) \to \Ext^{-1/2}(C^*_r(G),B)$ is an isomorphism.
\begin{proof} It follows from Theorem \ref{main} that the map
  $\Ext^{-1}(C^*_r(G),B) \to \Ext^{-1/2}(C^*_r(G),B)$ is
  surjective. To conclude that the map is also injective note that the
  six-term exact sequence of $K$-theory arising from an asymptotically split
  extension has trivial boundary maps and the resulting group
  extensions are split. Hence the injectivity of the map we consider will follow if we can
  show that $C^*_r(G)$ satisfies the UCT. Since $G$ is K-amenable $C^*_r(G)$
  is $KK$-equivalent to $C^*(G)$, cf. \cite{C}, so we may as well
  show that $C^*(G)$ satisfies the UCT. We do this in the following
  three steps: Since the class of $C^*$-algebras which satisfies the
  UCT is closed under countable inductive limits we need only show
  that $ C^*\left(\star_{i \leq n} G_i\right)$ satisfies the UCT. Next
  observe that it follows from \cite{Th3} that an amalgamated free
  product $A \star_{\mathbb C} B$ of unital separable $C^*$-algebras
  $A$ and $B$ is KK-equivalent to the mapping cone of the inclusion
  $\mathbb C \subseteq A \oplus B$. Thus $A\star_{\mathbb C} B$ will
  satisfy the UCT when $A$ and $B$ do. Since
$$
C^*\left(\star_{i \leq n} G_i\right) \simeq C^*\left(G_1\right)
\star_{\mathbb C} C^*\left(G_2\right)\star_{\mathbb C} \dots
\star_{\mathbb C} C^*(G_n) 
$$
we can apply this observation $n -1$ times to conclude that $
C^*\left(\star_{i \leq n} G_i\right)$ satisfies the UCT if each
$C^*(G_i)$ does. And this follows from \cite{Tu} because $G_i$ is amenable.
\end{proof}
\end{thm}

\end{document}